\documentclass[11pt]{amsart}

\usepackage[T1]{fontenc}
\usepackage[utf8]{inputenc}
\usepackage{lmodern}
\usepackage{microtype}
\usepackage[a4paper,margin=1.08in]{geometry}
\usepackage{amsmath,amssymb,amsthm,mathtools}
\usepackage{bm}
\usepackage{enumitem}
\usepackage{booktabs}
\usepackage{array}
\usepackage{aliascnt}
\usepackage{etoolbox}
\usepackage{placeins}
\usepackage{tikz}
\usetikzlibrary{decorations.markings, arrows.meta}
\usetikzlibrary{arrows.meta}
\usepackage{xcolor}
\usepackage{hyperref}
\usepackage[nameinlink,noabbrev]{cleveref}
\usepackage{graphicx}

\definecolor{linkblue}{RGB}{27,77,135}
\definecolor{wallred}{RGB}{165,48,48}
\definecolor{regionblue}{RGB}{49,102,166}
\definecolor{regiongold}{RGB}{184,126,27}
\definecolor{mygreen}{RGB}{50, 108, 98}
\hypersetup{
  colorlinks=true,
  linkcolor=linkblue,
  citecolor=linkblue,
  urlcolor=linkblue,
  pdftitle={Explicit Hamiltonian Classification in the F4(0) Toric Degeneration of CP2},
  pdfauthor={Han Lou and Shuo Zhang},
  pdfsubject={Lagrangian tori, toric degeneration, Hamiltonian isotopy}
}

\setlist[itemize]{leftmargin=2em,itemsep=0.25em,topsep=0.35em}
\setlist[enumerate]{leftmargin=2em,itemsep=0.35em,topsep=0.35em}
\allowdisplaybreaks
\numberwithin{equation}{section}

\newtheorem{theorem}{Theorem}[section]
\newtheorem*{maintheorem}{Main Theorem}
\newaliascnt{proposition}{theorem}
\newtheorem{proposition}[proposition]{Proposition}
\aliascntresetthe{proposition}
\newaliascnt{lemma}{theorem}
\newtheorem{lemma}[lemma]{Lemma}
\aliascntresetthe{lemma}
\newaliascnt{corollary}{theorem}
\newtheorem{corollary}[corollary]{Corollary}
\aliascntresetthe{corollary}
\newaliascnt{claim}{theorem}

\aliascntresetthe{claim}
\theoremstyle{definition}
\newaliascnt{definition}{theorem}
\newtheorem{definition}[definition]{Definition}
\aliascntresetthe{definition}
\newaliascnt{remark}{theorem}
\newtheorem{remark}[remark]{Remark}
\aliascntresetthe{remark}
\newtheorem*{remark*}{Remark}

\crefname{theorem}{theorem}{theorems}
\Crefname{theorem}{Theorem}{Theorems}
\crefname{proposition}{proposition}{propositions}
\Crefname{proposition}{Proposition}{Propositions}
\crefname{lemma}{lemma}{lemmas}
\Crefname{lemma}{Lemma}{Lemmas}
\crefname{corollary}{corollary}{corollaries}
\Crefname{corollary}{Corollary}{Corollaries}
\crefname{claim}{claim}{claims}
\Crefname{claim}{Claim}{Claims}
\crefname{definition}{definition}{definitions}
\Crefname{definition}{Definition}{Definitions}
\crefname{remark}{remark}{remarks}
\Crefname{remark}{Remark}{Remarks}

\newcommand{\CP}{\mathbb{CP}}
\newcommand{\R}{\mathbb{R}}

\newcommand{\Z}{\mathbb{Z}}

\newcommand{\Int}{\operatorname{Int}}

\newcommand{\im}{\operatorname{Im}}

\newcommand{\std}{\mathrm{std}}

\newcommand{\multiset}[1]{\{\!\{#1\}\!\}}

\AtEndEnvironment{theorem}{%
  \ifnum\value{section}=3\relax
    \label{thm:offwall}%
    \label{cor:offwall-triples}%
  \fi
}

\title[Explicit classification in the \(F_4(0)\) degeneration]
{Explicit Hamiltonian Classification in the
\(F_4(0)\) Toric Degeneration of \(\CP^2\)}
\author{Han Lou}
\address{Beijing International Center for Mathematical Research, Peking University}
\email{hlou423@gmail.com}
\author{Shuo Zhang}
\address{Morningside Center, Chinese Academy of Science}
\email{felix.zhang@berkeley.edu}
\date{}

\begin{document}

\begin{abstract}
We give an explicit coordinate description of the Hamiltonian isotopy
classes of the regular Lagrangian torus fibers of the smoothing
\(\widehat{F}_4(0)\) of the \(F_4(0)\) toric degeneration, expressed in
the explicit Oakley--Usher coordinates on
\(\CP^2(\sqrt2)\). For the wall fibers, they are not Hamiltonian isotopic
to standard toric fibers, and no two
distinct wall fibers are Hamiltonian isotopic.  For the off-wall
fibers, we find the standard toric fibers they are Hamiltonian isotopic to.
\end{abstract}
\maketitle

\tableofcontents

\section{Introduction}

Wu constructed an exotic monotone Lagrangian torus in \(\CP^2\) from a
smoothing of the singular toric surface \(F_4(0)\) similar to the degeneration used by
Fukaya--Oh--Ohta--Ono for \(S^2\times S^2\)
\cite{FOOO,WuExotic}.
Albers--Frauenfelder descended a torus in $T^*S^2$ to one in $T^*\mathbb{R}P^2$, which then includes into $\mathbb{CP}^2$ as a monotone torus \cite{AlbersFrauenfelder}. 
Chekanov--Schlenk constructed a twist torus in $B^4(\sqrt{2})$, which gives a torus in $\mathbb{CP}^2$ under the dense symplectic embedding $B^4(\sqrt{2})\hookrightarrow\mathbb{CP}^2$ \cite{ChekanovSchlenk}.
Biran--Cornea constructed a torus by the result of the Biran circle bundle construction using an equatorial circle in the quadric $Q_2(\sqrt{2})$ \cite{BiranCornea}.

Oakley and Usher gave a symplectomorphism from the smoothing
\(\widehat{F}_4(0)\) to \(\CP^2(\sqrt2)\) and identified Wu's monotone
fiber in their model
\cite[Proposition 3.3]{OakleyUsher}.  They also proved that this fiber is
Hamiltonian isotopic to the Albers--Frauenfelder torus,  Chekanov--Schlenk torus, and the Biran--Cornea torus
\cite[Theorem 1.2]{OakleyUsher}. Gadbled also proved in \cite{Gadbled} that Chekanov--Schlenk torus and Biran--Cornea torus are Hamiltonian isotopic.

Lou classified the regular fibers of the \(F_2(0)\) degeneration of
\(S^2\times S^2\) up to Hamiltonian isotopy \cite{Lou}. We consider the corresponding question
for every regular fiber of the smoothing of $F_4(0)$.

Throughout the paper, let $\mathbb{CP}^2(\sqrt{2})$ be the coisotropic reduction of the sphere of radius $\sqrt{2}$ in $\mathbb{C}^3$ with the Fubini-Study form $\omega_{\mathrm{FS}}$ such that $\int_{\mathbb{CP}^1}\omega_{\mathrm{FS}}=2\pi$. There is a toric structure on $\mathbb{CP}^2(\sqrt{2})$ with moment map 
\begin{align*}
    \mu_{\mathrm{std}}: \mathbb{CP}^2(\sqrt{2}) &\to \mathbb{R}^2\\
    [z_0: z_1: z_2] & \mapsto\left(\frac{1}{2}|z_1|^2, \frac{1}{2}|z_2|^2\right)
\end{align*}
The moment polytope $\Delta_{\mathrm{std}}$ is 
\begin{align*}
    \Delta_{\mathrm{std}}=\left\{(x, y)\in\mathbb{R}^2\mid x\ge 0, y\ge 0, x+y\le 1\right\}
\end{align*}
Given an interior point $(x, y)$ of $\Delta_{std}$, $\mu_{\mathrm{std}}^{-1}(x, y)$ is a Lagrangian torus which is called a standard toric fiber and denoted by $T(x, y)$.

For the construction of $\widehat{F}_4(0)$, as in \cite{WuExotic}, one begins with a symplectic toric orbifold that is denoted $F_4(0)$ and whose moment polytope is 
\begin{align*}
    \Delta_W=\left\{(x, y)\in \mathbb{R}^2\mid 0\le x\le 2, 0\le y\le \frac{1}{2}-\frac{1}{4}x\right\}
\end{align*}
with exactly one singular point sitting over the point $\left(0, \frac{1}{2}\right)\in\Delta_W$. Then by replacing a neighborhood of the singular point with a neighborhood of the zero section of the cotangent bundle $T^*\mathbb{R}P^2$, one obtains a manifold denoted $\widehat{F}_4(0)$ that is symplectomorphic to $\mathbb{CP}^2(\sqrt{2})$. The moment polytope of $\widehat{F}_4(0)$ is still $\Delta_W$ and we denote the moment map by 
\begin{align*}
    \mu_W: \widehat{F}_4(0)\to \Delta_W
\end{align*}
Then $\mu_W^{-1}\left(0, \frac{1}{2}\right)=\mathbb{R}P^2$. Let $(x, y)$ be an interior point of $\Delta_W$. Denote by $L(x, y):=\mu_W^{-1}(x, y)$, which is a Lagrangian torus in $\widehat{F}_4(0)$.

\begin{figure}[h]
    \centering
    \includegraphics[scale=0.7]{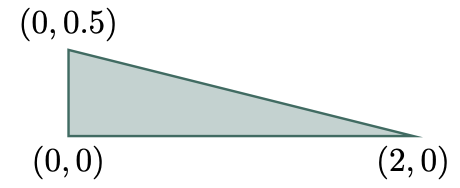}
    \caption{The moment polytope $\Delta_{W}$ of $\widehat{F}_4(0)$}
\end{figure}

We have the following result. 

\begin{maintheorem}\label{thm:main}
    For every \((x,y)\in\Int\Delta_W\), the following hold.
\begin{enumerate}
\item If \(x+2y<1\), then $L(x,y)$ is Hamiltonian isotopic to $T\left(y,1-x-y\right)$.

\item If \(x+2y>1\), then $L(x,y)$ is Hamiltonian isotopic to $T\left(x+3y-1,y\right)$.

\item If \(x+2y=1\), then \(L(x,y)\) is not Hamiltonian isotopic to any
standard toric fiber.
\end{enumerate}
\end{maintheorem}

\begin{remark}
    We think it is also possible to use almost toric fibration and the results in \cite{ShelukhinTonkonogVianna} to get the same result.
\end{remark}

For context, the Hamiltonian distinction studied here is finer than
Lagrangian isotopy: every Lagrangian torus in \(\CP^2\) is Lagrangian
isotopic to the Clifford torus
\cite[Theorem A]{DimitroglouRizellGoodmanIvrii}.
Vianna also constructed infinitely many Hamiltonian isotopy classes of
monotone Lagrangian tori \cite[Theorem 1.1]{Vianna}.

\begin{remark}[Independence of the chosen symplectomorphism]
\label{rem:independence-identification}
The Hamiltonian isotopy classification above does not depend on the
chosen symplectomorphism from $\widehat{F}_4(0)$ to \(\mathbb{CP}^2(\sqrt{2})\).  
By Gromov's connectedness theorem
\cite{GromovPseudoHolomorphic}, the symplectomorphism group of
\((\CP^2(\sqrt{2}),\omega_{\mathrm{FS}})\) is connected.  Moreover, \(H^1(\CP^2(\sqrt{2});\R)=0\),
so the closed one-form
\(\iota_{X_t}\omega_\mathrm{FS}\) associated to the generating vector field of any
symplectic isotopy is exact at every time.  Hence every symplectic
isotopy on \(\CP^2(\sqrt{2})\) is Hamiltonian.
\end{remark}

By displacement energy germ we can get the following corollary
\begin{corollary}
    If $L(x, y)$ is Hamiltonian isotopic to $L(x^\prime, y^\prime)$ for $x+2y=x^\prime+2y^\prime=1$, then $(x, y)=(x^\prime, y^\prime)$
\end{corollary}

\begin{remark}
    As in \cite{McDuffProbes}, we know the Lagrangian tori $L(x, y)$ for $x+2y=1$ and $\frac{1}{2}<x<1$ are displaceable by probes. Combining with the Main Theorem, they are a continuum of displaceable Lagrangian tori but not Hamiltonian isotopic to standard toric fibers in $\CP^2$.  It is still an open question if the Lagrangian tori $L(x, y)$ for $x+2y=1$ and $0<x<\frac{1}{3}$ are displaceable. 
\end{remark}

\section{Symplectic preliminaries}
\label{sec:symplectic-preliminaries}

\subsection{Hamiltonian reduction and lifting}
\label{subsec:reduction-and-lifting}

We record the support-localized lifting statement used in Section \ref{sec:OU-model}. This is a more general result than what we need to prove the main theorem.

\begin{lemma}[Lifting an isotopy from a reduced surface]
\label{lem:lifting-reduced-isotopy}
Let $G\colon M\to\mathbb{R}$ be a smooth function on a symplectic four-manifold
$(M,\omega_M)$ whose Hamiltonian flow generates an $S^1$-action. Suppose that the action is
free on $G^{-1}(p)$, where $p\in\mathbb{R}$. Write
\[
\iota\colon G^{-1}(p)\hookrightarrow M,
\qquad
\pi\colon G^{-1}(p)\longrightarrow G^{-1}(p)/S^1
\]
for the inclusion and quotient maps, and let $\omega_{\mathrm{red}}$ be the reduced
symplectic form characterized by
\[
\pi^*\omega_{\mathrm{red}}=\iota^*\omega_M.
\]
Let $\Gamma_0,\Gamma_1\subset G^{-1}(p)/S^1$ be compact embedded circles related by a
compactly supported Hamiltonian isotopy $\overline{\phi}_t$, and set
\[
\mathcal T:=\bigcup_{t\in[0,1]}\overline{\phi}_t(\Gamma_0).
\]
Then the Lagrangian tori $\pi^{-1}(\Gamma_0)$ and $\pi^{-1}(\Gamma_1)$ are related by
a Hamiltonian isotopy of $M$, supported in an arbitrarily small invariant neighborhood of
$\pi^{-1}(\mathcal T)$.
\end{lemma}

\begin{proof}
This is the circle-reduction and support-localized version of
\cite[Lemma~3.1, p.~3861]{AbreuMacarini}. The same lifting construction, including extension
by a cutoff, appears in the proof of
\cite[Theorem~3.2]{BrendelClassification}; see also
\cite[Remark~3.4]{BrendelClassification} for the support refinement. We give the details needed
in the present noncompact setting.

Let $U$ be any $S^1$-invariant open neighborhood of $\pi^{-1}(\mathcal T)$ in $M$. Note that
$\mathcal T$ is compact. Since
\[
\pi\colon G^{-1}(p)\longrightarrow G^{-1}(p)/S^1
\]
is a principal $S^1$-bundle, $\pi^{-1}(\mathcal T)$ is compact. Choose relatively compact
open neighborhoods $V$ and $V_0$ of $\mathcal T$ in $G^{-1}(p)/S^1$ such that
\[
\bar V_0\subset V, \quad \pi^{-1}(\bar V)\subset U\cap G^{-1}(p).
\]
By averaging over $S^1$ we can get an $S^1$-invariant Riemannian metric $g$ on $M$. Then define
\[
Y=\frac{\nabla G}{\|\nabla G\|^2}.
\]
Then $Y$ is an $S^1$-invariant vector field and 
\begin{align*}
g(Y, X_G)=\frac{dG(X_G)}{||\nabla G||^2}=0
\end{align*}
where $X_G$ is the Hamiltonian vector field of $G$. Thus $Y$ is transverse to $G^{-1}(p)$. 
 Since $\pi^{-1}(\bar V)$ is compact, there is $\epsilon>0$ such that the flow $\phi_Y^s$ of $Y$ defines an $S^1$-invariant collar neighborhood $\mathcal{N}$ of $\pi^{-1}(V)$
 \begin{align*}
     \Phi: \pi^{-1}(V)\times (-\epsilon, \epsilon) &\to \mathcal{N}\subset \mathcal{U}\\
     (z, s) &\mapsto \phi_Y^s(z)
 \end{align*}

Next we choose two cutoff functions. Let $\beta$ be a compactly supported smooth function on $G^{-1}(p)/S^1$  such that 
 \begin{align*}
     0\le \beta\le 1, \quad \beta|_{V_0}=1, \quad \mathrm{supp}(\beta)\subset V
 \end{align*}
Let $\chi$ be a compactly supported smooth function on $(-\epsilon, \epsilon)$ such that 
\begin{align*}
    0\le\chi\le 1, \quad \chi=1\: \text{on a neighborhood of}\:0
\end{align*}

Finally we define $K_t: M\to\mathbb{R}$ on $\mathcal{N}$ by
\begin{align*}
    K_t(\Phi(z, s))=\chi(s)\beta(\pi(z))\bar K_t(\pi(z))
\end{align*}
where $\bar K_t$ is the Hamiltonian generating $\bar \phi_t$, and extend it to $M$ by $0$. We have 
\begin{align*}
    K_t|_{\pi^{-1}(V_0)}=(\pi^*\bar K_t)|_{\pi^{-1}(V_0)}
\end{align*}
Since $K_t$ is $S^1$-invariant, $\{K_t, G\}=0$ and the Hamiltonian flow $\phi_K^t$ of $K_t$ preserves the level set $G^{-1}(p)$. 

Let $v\in T_z(\pi^{-1}(V_0))=T_z(G^{-1}(p))$ for $z\in V_0$.
Using $\pi^*\omega_{\mathrm{red}}=\iota^*\omega_M$, we have
\begin{align*}
    \omega_{\mathrm{red}}(\pi_*X_{K_t}, \pi_*v)=\omega_{\mathrm{red}}(X_{\bar K_t}, \pi_*v)
\end{align*}
where $X_{K_t}$ and $X_{\bar K_t}$ are the Hamiltonian vector fields of $K_t$ and $\bar K_t$ respectively. Thus
\begin{align}
    \pi_*X_{K_t}=X_{\bar K_t}
\end{align}
Then we have the following relation for their flows
\begin{align}
    \pi\circ \phi_K^t=\bar \phi_t\circ \pi
\end{align}
Applying the above relation to $\pi^{-1}(\Gamma_0)$, we have 
\begin{align*}
    \phi_K^t(\pi^{-1}(\Gamma_0))=\pi^{-1}(\bar \phi_t(\Gamma_0))
\end{align*}
Furthermore, 
\begin{align*}
    \phi_K^1(\pi^{-1}(\Gamma_0))=\pi^{-1}(\bar \phi_1(\Gamma_0))=\pi^{-1}(\Gamma_1)
\end{align*}
\end{proof}

\subsection{Toric fibers and displacement-energy germs}
\label{subsec:displacement-energy-germs}

Let $(M, \omega_M)$ be a symplectic manifold. For a compactly supported Hamiltonian $H\colon[0,1]\times M\to\mathbb{R}$, its Hofer norm
is
\[
\|H\|_{\mathrm H}
=
\int_0^1
\left(
\max_M H_t-\min_M H_t
\right)\,dt.
\]
For a compact subset $K\subset M$, the displacement energy of $K$ is defined as 
\[
e_M(K)
=
\inf\left\{
\|H\|_{\mathrm H}:
\phi_H^1(K)\cap K=\varnothing
\right\},
\]
with $e_M(K)=+\infty$ when $K$ is nondisplaceable.

For each closed embedded Lagrangian submanifolds in $(M, \omega_M)$ by Weinstein's neighborhood theorem, there is a symplectomorphism from a neighborhood of the zero section of $T^*L$ to a neighborhood of $L$ in $M$ such that the zero section is mapped to $L$. Then Chekanov--Schlenk defined the displacement energy germ as following.

\begin{definition}[Displacement-energy germ]\cite{ChekanovSchlenk}
\label{def:displacement-energy-germ}
The \textit{displacement energy germ} is a function germ 
\begin{align*}
S_L^e: H^1(L; \mathbb{R}) &\to [0, +\infty]\\
L_{\xi} &\mapsto e_M(L_{\xi})
\end{align*}
at the point $0\in H^1(L; \mathbb{R})$, where $\xi\in H^1(L; \mathbb{R})$ is sufficiently small and $L_\xi$ is the image of a closed $1$-form on $L$ representing the class $\xi$. 

Displacement energy germs are symplectically invariant in the following sense: for each symplectomorphism $\psi$ we have
\begin{align*}
    S_{\psi(L)}^e=S_L^e\circ(\psi|_L)^*
\end{align*}
\end{definition}

\begin{definition}[Integral-affine lengths and facet distances]
\label{def:facet-distances}
An affine coordinate change on \(\R^2\) is \emph{integral affine} if its
linear part lies in \(GL(2,\Z)\).  If a line segment has rational
direction and the difference of its endpoints is
\(\pm\ell\nu\), where \(\ell\geq0\) and
\(\nu\in\Z^2\) is primitive, then \(\ell\) is its
\emph{integral-affine length}.

Let
\[
 \Delta
 =
 \{u\in\R^2:
 \ell_i(u)=\langle\nu_i,u\rangle+\kappa_i\geq0,\
 i=1,\ldots,N\}
\]
be a Delzant polygon, where \(\nu_i\in\Z^2\) is the primitive inward
normal to the facet \(F_i=\{\ell_i=0\}\).  For
\(u\in\Int\Delta\), the number \(\ell_i(u)\) is the
\emph{integral-affine distance} from \(u\) to \(F_i\).  Primitivity of
\(\nu_i\) fixes its scale; in particular, this is not Euclidean
distance.  The \emph{facet-distance multiset} of \(u\), or of the toric
fiber over \(u\), is
\[
 \mathfrak d_\Delta(u)
 =
 \multiset{\ell_1(u),\ldots,\ell_N(u)},
\]
where the braces denote an unordered multiset and repetitions are
retained.

For the standard moment triangle
\[
 \Delta_{\std}
 =
 \{(A,B)\in\R^2:A\geq0,\ B\geq0,\ 1-A-B\geq0\},
\]
the ordered facet-distance triple is
\begin{equation}\label{eq:nearby-wall-facet-distances}
 (A,B,C),\qquad C=1-A-B,
\end{equation}
and the corresponding multiset is \(\multiset{A,B,C}\).  In the
normalization \(\int_{\CP^1}\omega_M=2\pi\), the three associated basic
disk areas are \(2\pi A,2\pi B,2\pi C\).
\end{definition}

Although it is a well known result we also include the computation of the displacement energy of a standard toric fiber in $\CP^2(\sqrt{2})$ for the convenience of the readers.
We will use symmetric probes introduced by Abreu-Borman-McDuff in \cite{AbreuBormanMcDuff}, generalizing the definition of probes introduced by McDuff in \cite{McDuffProbes}.
\begin{definition}[Symmetric probe]
    A \textit{probe} $P$ in a rational polytope $\Delta\subset \mathbb{R}^n$ is a directed rational line segment contained in $\Delta$ whose initial point $b_P$ lies in the interior of a facet $F_P$ of $\Delta$ and whose direction vector $v_P\in\mathbb{Z}^n$ is primitive and integrally transverse to the facet $F_P$. A probe $P$ is \textit{symmetric} if the endpoint $e_P$ lies on the interior of a facet $F_P^\prime$ that is integrally transverse to $v_P$.
\end{definition}
\begin{proposition}\label{dis}
    The displacement energy of the standard toric fiber $T(a, b)$ is 
    \begin{align*}
    e_{\CP^2(\sqrt{2})}(T(a, b))=\min\{a, b, 1-a-b\}
    \end{align*}
    for $(a, b)\in\Delta_{\mathrm{std}}$ and $(a, b)\ne\left(\frac{1}{3}, \frac{1}{3}\right)$.
\end{proposition}
\begin{proof}
    If $a$, $b$, and $1-a-b$ are distinct from each other, then the result is \cite[Lemma 3.6]{BrendelVersal}.

    If $a=1-a-b$, denote the point $(a, b)$ by $A$. Consider the symmetric probe $P$ $x=a$. The intersection points of $P$ with the facet $\{y=0\}$ and the facet $\{1-x-y=0\}$ are $B(a, 0)$ and $C(a, 1-a)$ respectively. Then the distance $d(A, B)$ between points $A$ and $B$ is $b$ and the distance $d(A, C)$ between the points of $A$ and $C$ is $1-a-b$. Thus by \cite[Proposition 3.4]{BrendelVersal} we have 
    \begin{align*}
        e_{\CP^2(\sqrt{2})}(T(a, b))\le 
        \begin{cases}
            b & \mathrm{for}\: b<1-a-b\\
            1-a-b & \mathrm{for}\: b>1-a-b\\
        \end{cases}
    \end{align*}
    In either case we have $e_{\CP^2(\sqrt{2})}(T(a, b))\le\min\{a, b, 1-a-b\}$. On the other hand by \cite[Proposition 3.2]{BrendelVersal}, $e_{\CP^2(\sqrt{2})}(T(a, b))\ge\min\{a, b, 1-a-b\}$. Thus
    \begin{align*}
        e_{\CP^2(\sqrt{2})}(T(a, b))=\min\{a, b, 1-a-b\}
    \end{align*}
    \begin{figure}[h]
    \centering
    \includegraphics[scale=0.7]{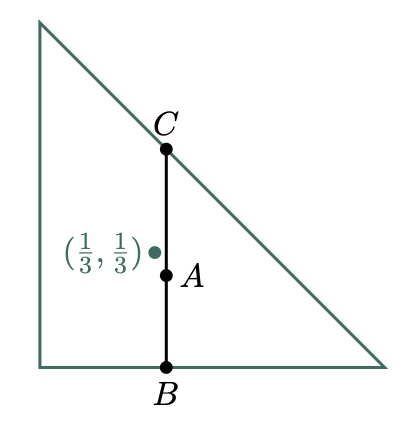}
    \caption{An illustration of the symmetric probe $x=a$}
\end{figure}

If $b=1-a-b$, then the proof is the same but use the symmetric probe $y=b$. If $a=b$, then we can use either the symmetric probe $x=a$ or $y=b$.
\end{proof}

We will also use the following complete classification.

\begin{theorem}[Brendel]
\label{thm:Brendel-classification}
Two standard toric fibers in $\mathbb{CP}^2(\sqrt{2})$, with facet-distance triples
$(A,B,C)$ and $(A',B',C')$, are Hamiltonian isotopic if and only if
\[
\{\!\{A,B,C\}\!\}
=
\{\!\{A',B',C'\}\!\}.
\]
\end{theorem}

\begin{proof}
This is an equivalent formulation of \cite[Proposition~5.4]{BrendelClassification} and the fact that an integral symmetry of $\Delta_{\mathrm{std}}$ is a permutation of the three facets. For completeness we prove that this theorem and the first part of \cite[Proposition~5.4]{BrendelClassification} are equivalent.

For a point \(u=(x,y)\in\Delta_{\mathrm{std}}\), its three facet distances are
\[
\ell(u)
=
\bigl(\ell_{1}(u),\ell_{2}(u),\ell_{3}(u)\bigr)
=
\bigl(x,y,1-x-y\bigr).
\]
Thus, if
\[
\ell(u)=(A,B,C),
\]
then
\[
A=x,\qquad B=y,\qquad C=1-x-y,
\]
and therefore
\[
u=(A,B),
\qquad
A+B+C=1.
\]
Hence the map
\[
\ell\colon \Delta_{\mathrm{std}}
\longrightarrow
\left\{
(A,B,C)\in\mathbb{R}_{\geq 0}^{3}
\;\middle|\;
A+B+C=1
\right\}
\]
is injective, with inverse
\[
\ell^{-1}(A,B,C)=(A,B).
\]

For each permutation \(\sigma\in S_{3}\), let
\[
P_{\sigma}(A_{1},A_{2},A_{3})
=
\bigl(A_{\sigma(1)},A_{\sigma(2)},A_{\sigma(3)}\bigr)
\]
and define
\[
g_{\sigma}
=
\ell^{-1}\circ P_{\sigma}\circ\ell.
\]
Then,
\[
\ell\bigl(g_{\sigma}(x)\bigr)
=
P_{\sigma}\bigl(\ell(x)\bigr).
\]
Since \(P_{\sigma}\) merely permutes three nonnegative numbers whose sum
is $1$, each \(g_{\sigma}\) preserves \(\Delta_{\mathrm{std}}\) and can easily be verified to be  integral-affine symmetries of
\(\Delta_{\mathrm{std}}\).

Conversely, every affine symmetry of the triangle
\(\Delta_{\mathrm{std}}\) permutes its three vertices. Since an affine map is
uniquely determined by its values on the three vertices, every
integral-affine symmetry of \(\Delta_{\mathrm{std}}\) is equal to one of the
maps \(g_{\sigma}\). 

Therefore,
\[
v=g_{\sigma}(u)
\quad\text{for some }\sigma\in S_{3}
\]
if and only if
\[
\ell(v)
=
P_{\sigma}\bigl(\ell(u)\bigr)
\quad\text{for some }\sigma\in S_{3}.
\]

Now write
\[
\ell(u)=(A,B,C),
\qquad
\ell(v)=(A',B',C').
\]
The preceding equivalence becomes
\[
v=g(u)
\quad\text{for some integral symmetry } g \text{ of }\Delta_{\mathrm{std}}
\]
if and only if
\[
(A',B',C')
\]
is a permutation of
\[
(A,B,C).
\]
Equivalently,
\[
v=g(u)
\quad\text{for some integral symmetry }g\text{ of }\Delta_{\mathrm{std}}
\quad\Longleftrightarrow\quad
\{\!\{A,B,C\}\!\}
=
\{\!\{A',B',C'\}\!\}.
\]
This proves that Theorem \ref{thm:Brendel-classification} and the first part of \cite[Proposition 5.4]{BrendelClassification} are equivalent.
\end{proof}

\section{The off-wall fibers}
\label{sec:OU-model}

We define $\mathbb{CP}^2(\sqrt{2})$ as the coisotropic reduction of the sphere of radius $\sqrt{2}$ in $\mathbb{C}^3$. 
As in \cite{WuExotic}, one begins with a symplectic toric orbifold that is denoted $F_4(0)$ and whose moment polytope is 
\begin{align*}
    \Delta_W=\left\{(x, y)\in \mathbb{R}^2\mid 0\le x\le 2, 0\le y\le \frac{1}{2}-\frac{1}{4}x\right\}
\end{align*}
with exactly one singular point sitting over the point $\left(0, \frac{1}{2}\right)\in\Delta_W$. Then by replacing a neighborhood of the singular point with a neighborhood of the zero section of the cotangent bundle $T^*\mathbb{R}P^2$, one obtains a manifold denoted $\widehat{F}_4(0)$ that is symplectomorphic to $\mathbb{CP}^2(\sqrt{2})$. The moment polytope of $\widehat{F}_4(0)$ is still $\Delta_W$ and we denote the moment map by 
\begin{align*}
    \mu_W: \widehat{F}_4(0)\to \Delta_W
\end{align*}
Then $\mu_W^{-1}\left(0, \frac{1}{2}\right)=\mathbb{R}P^2$. Let $(x, y)$ be an interior point in $\Delta_W$, then $\mu_W^{-1}(x, y)$ is a Lagrangian torus, denoted by $L(x, y)$.

In \cite{OakleyUsher}, J. Oakley and M. Usher gave an explicit symplectomorphism between $\widehat{F}_4(0)$ and $\mathbb{CP}^2(\sqrt{2})$. Under this symplectomorphism, the image of $L(x, y)$, still denoted by $L(x, y)$, is
\begin{align}\label{eq:fiber-equations-xy}
 L(x,y)=& \left\{[z_0: z_1: z_2]\in\mathbb{CP}^2(\sqrt{2})\mid \frac{1}{2}\sqrt{4-\left|\sum_{j=0}^2z_j^2\right|^2}+\im(\bar z_1z_2)=x, \frac{1}{2}-\frac{1}{4}\sqrt{4-\left|\sum_{j=0}^2z_j^2\right|^2}=y\right\}\\
 =&\left\{
 [z_0: z_1: z_2]\in \mathbb{CP}^2(\sqrt{2})\mid 
 |z_0^2+z_1^2+z_2^2|=4\sqrt{y(1-y)},\
 \im(\bar z_1z_2)=x+2y-1
 \right\}.
\end{align}

Then we change the coordinates by
\begin{equation}\label{eq:pq-def}
 p=x+2y-1,\qquad q=1-2y.
\end{equation}
Under the coordinates $(p, q)$, the moment polytope $\Delta_W$ becomes 
\begin{align*}
    \left\{(p, q)\in\mathbb{R}^2\mid -q\le p\le q, 0\le q\le 1\right\}
\end{align*}
and we still call it $\Delta_W$. The Lagrangian torus $L(x, y)$ can be written as 
\begin{align*}
    \widetilde{L}(p, q)=\left\{[z_0: z_1: z_2]\in\mathbb{CP}^2(\sqrt{2})\mid |z_0^2+z_1^2+z_2^2|=2\sqrt{1-q^2}, \im(\bar z_1z_2)=p\right\}
\end{align*}

\begin{figure}[h]
    \centering
    \includegraphics[scale=0.7]{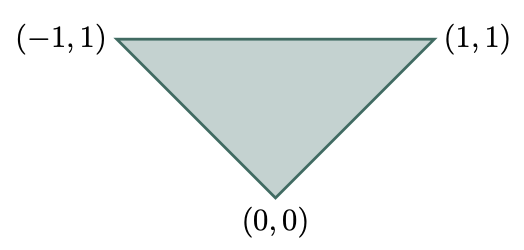}
    \caption{The moment polytope $\Delta_W$ under the $(p, q)$ coordinates}
\end{figure}



Consider the Darboux chart $(U_0, \psi^{-1})$ on $\mathbb{CP}^2(\sqrt{2})$ where 
\begin{align*}
    U_0=\left\{[z_0: z_1: z_2]\in\mathbb{CP}^2\mid z_0\ne 0\right\}
\end{align*}
and 
\begin{align*}
    \psi^{-1}: U_0 &\to B^4(\sqrt{2})\\
    [z_0: z_1: z_2] &\mapsto \left(\frac{z_1|z_0|}{z_0}, \frac{z_2|z_0|}{z_0}\right)
\end{align*}
Then by a routine computation we have
\begin{align*}
    \psi: B^4(\sqrt{2}) &\to \mathbb{CP}^2(\sqrt{2})\\
    (w_1, w_2) &\mapsto \left[\sqrt{2-|w_1|^2-|w_2|^2}: w_1: w_2\right]
\end{align*}

\begin{lemma}[Ball containment]\label{lem:ball-containment}
Every fiber \(\widetilde{L}(p, q)\) with \((p, q)\in\Int\Delta_W\) is contained in the
standard symplectic ball chart
\(\psi(B^4(\sqrt2))=\{z_0\neq0\}\).
\end{lemma}

\begin{proof}
Suppose a point of the fiber had \(z_0=0\).  Then
\(|z_1|^2+|z_2|^2=2\), and direct expansion gives
\begin{equation}\label{eq:boundary-identity}
 |z_1^2+z_2^2|^2+
 4\im(\bar z_1z_2)^2
 =
 (|z_1|^2+|z_2|^2)^2=4.
\end{equation}
Thus
\[
 4(1-q^2)+4p^2=4
\]
Thus \(p^2=q^2\), contradicting the strict interior condition
\(|p|<q\).
\end{proof}

Thus we have 
\begin{align*}
    L_1(p, q):&=\psi^{-1}\left(\widetilde{L}(p, q)\right)\\
    &=\left\{(w_1, w_2)\in B^4(\sqrt{2})\mid \left|2-|w_1|^2-|w_2|^2+w_1^2+w_2^2\right|=2\sqrt{1-q^2}, \im(\bar w_1w_2)=p\right\}
\end{align*}

Next we change the coordinates on $B^4(\sqrt{2})$ by
\begin{equation*}\label{eq:unitary-change}
 u_1=\frac{w_1-iw_2}{\sqrt2},\qquad
 u_2=\frac{w_1+iw_2}{\sqrt2}.
\end{equation*}
Then 
\begin{align*}
L_1(p, q)=\left\{(u_1, u_2)\in B^4(\sqrt{2})\mid \left|2-|u_1|^2-|u_2|^2+2u_1u_2\right|=2\sqrt{1-q^2}, \frac{|u_1|^2-|u_2|^2}{2}=p\right\}
\end{align*}


Now we use the method in \cite{EliashbergPolterovich} to describe $L_1(p, q)$ for $p\ne 0$. First we focus on the function
\begin{align*}
    G: B^4(\sqrt{2}) &\to \mathbb{R}\\
    (u_1, u_2) &\mapsto \frac{|u_1|^2-|u_2|^2}{2}
\end{align*}
Note that $L_1(p, q)\subset G^{-1}(p)$ and the Hamiltonian flow of $G$ gives following action $S^1$-action on $G^{-1}(p)$
\begin{align*}
    (e^{it}, e^{-it})\cdot (u_1, u_2)=(e^{it}u_1, e^{-it}u_2)
\end{align*}
As in \cite[Section 4.2]{EliashbergPolterovich}, this action is free on $B^4(\sqrt{2})\setminus\{0\}$. In particular it freely acts on $G^{-1}(p)$ for $p\ne 0$. By the Marsden-Weinstein-Meyer Theorem, there is a unique symplectic form $\omega_{\mathrm{red}}$ on $G^{-1}(p)/S^1$ such that 
\begin{align*}
    \iota^*\left(\frac{i}{2}du_1\wedge d\bar u_1+\frac{i}{2}du_2\wedge d\bar u_2\right)=\pi^*\omega_{\mathrm{red}}
\end{align*}
where $\iota: G^{-1}(p)\hookrightarrow B^4(\sqrt{2})$ is the inclusion and $\pi: G^{-1}(p)\to G^{-1}(p)/S^1$ is the projection.

\begin{proposition}
    For $p\ne 0$, there is a symplectomorphism 
    \begin{align*}
        W: \left(G^{-1}(p)/S^1, \omega_{\mathrm{red}}\right) &\to \left(D_p, \frac{1}{2\sqrt{p^2+x^2+y^2}}dx\wedge dy\right)\\
        [u_1, u_2] &\mapsto u_1u_2
    \end{align*}
    where $D_p=\left\{(x, y)\in\mathbb{R}^2\mid x^2+y^2<1-p^2\right\}$ and $[u_1, u_2]$ is the orbit of $(u_1, u_2)$ under the $S^1$-action.
\end{proposition}
\begin{proof}
    Assume $[u_1, u_2]=[v_1, v_2]$. Then there is a $t$ such that $v_1=e^{it}u_1$ and $v_2=e^{-it}u_2$. We have 
    \begin{align*}
        v_1v_2=e^{it}u_1e^{-it}u_2=u_1u_2
    \end{align*}
    Thus the map $W$ is well-defined.
    
    First we show that $W$ is a bijection. Since $|u_1|^2-|u_2|^2=2p$ and $|u_1|^2+|u_2|^2<2$, then $|u_2|^2<1-p$. Furthermore,
    \begin{align*}
        |u_1u_2|=|u_2|\sqrt{2p+|u_2|^2}<\sqrt{1-p^2}
    \end{align*}
    Thus the image of the map $W$ is contained in $D_p$. 

    Assume $W(u_1, u_2)=W(v_1, v_2)$, that is $u_1u_2=v_1v_2$. We need to show that $(u_1, u_2)$ and $(v_1, v_2)$ are in the same $S^1$-orbit.

    \textbf{Case 1}: $u_1=0$ and $u_2\ne 0$
    
     Since $u_1u_2=v_1v_2$, one of $v_1$ and $v_2$ has to be zero. If $v_2=0$, then $-|u_2|^2=2p\le0$ and $|v_1|^2=2p\ge0$. Thus $p=0$, contradicting with the condition $p\ne 0$. Thus $v_1=0$. Then $|u_2|^2=|v_2|^2=-2p$. There is an element $e^{it}\in S^1$ such that $v_2=e^{it}u_2$, that is 
     \begin{align*}
         (e^{-it}, e^{it})\cdot (0, u_2)=(0, v_2)
     \end{align*}

     \textbf{Case 2}: $u_1\ne 0$ and $u_2=0$

     The proof for Case 2 is the same with that for Case 1.

     \textbf{Case 3}: $u_1\ne 0$ and $u_2\ne 0$

     Using polar coordinates, $u_j=\sqrt{\xi_j}e^{i\theta_j}$ and $v_j=\sqrt{\zeta_j}e^{i\tau_j}$ for $j=1, 2$. Since $u_1u_2=v_1v_2$, then 
     \begin{align*}
     \sqrt{\xi_1\xi_2}e^{i(\theta_1+\theta_2)}=\sqrt{\zeta_1\zeta_2}e^{i(\tau_1+\tau_2)}
     \end{align*}
     Thus $\xi_1\xi_2=\zeta_1\zeta_2$ and $\theta_1+\theta_2=\tau_1+\tau_2$ modulo $2\pi$. Since $\xi_1=2p+\xi_2$ and $\zeta_1=2p+\zeta_2$, then 
     \begin{align*}
         \xi_2(2p+\xi_2)=\zeta_2(2p+\zeta_2)
     \end{align*}
     that is 
     \begin{align*}
         (\xi_2-\zeta_2)(\xi_2+\zeta_2+2p)=0
     \end{align*}
     If $\xi_2+\zeta_2+2p=0$, then $\xi_1+\zeta_2=0$ since $\xi_1-\xi_2=2p$. Since $\xi_1\ge 0$ and $\zeta_2\ge 0$, then $\xi_1=\zeta_2=0$, i.e. $u_1=v_2=0$, which is not possible as in Case 1.

    Thus $\xi_2-\zeta_2=0$. Furthermore, $\xi_1=\zeta_1$.

    Since $\theta_1+\theta_2=\tau_1+\tau_2$ $2\pi$, let $\theta=\tau_1-\theta_1=\theta_2-\tau_2$. Thus
    \begin{align*}
        (e^{i\theta}, e^{-i\theta})\cdot (\sqrt{\xi_1}e^{i\theta_1}, \sqrt{\xi_2}e^{i\theta_2})=(\sqrt{\zeta_1}e^{i\tau_1}, \sqrt{\zeta_2}e^{i\tau_2})
    \end{align*}

    Now we show that the map $W$ is surjective. Given $w\in D_p$ and $w\ne 0$, denote $w=re^{i\theta}$ using the polar coordinates. Then 
    \begin{align*}
        W\left(\left[\sqrt{p+\sqrt{p^2+r^2}}e^{i\theta}, \sqrt{\sqrt{p^2+r^2}-p}\right]\right)=w
    \end{align*}
    
    If $w=0$ and $p>0$, then $W([u_1, 0)]=0$ for some $u_1$ with $|u_1|^2=2p$. If $w=0$ and $p<0$, then $W([0, u_2])=0$ for some $u_2$ with $|u_2|^2=-2p$.

    Next we show 
    \begin{align*}
        W^*\left(\frac{1}{2\sqrt{p^2+x^2+y^2}}dx\wedge dy\right)=\omega_{\mathrm{red}}
    \end{align*}
    Since $\omega_{\mathrm{red}}$ is unique, we only need to show 
    \begin{align*}
        \pi^*W^*\left(\frac{1}{2\sqrt{p^2+x^2+y^2}}dx\wedge dy\right)=\pi^*\omega_{\mathrm{red}}
    \end{align*}
    Using polar coordinates, we write $u_1=r_1e^{i\theta_1}$ and $u_2=r_2e^{i\theta_2}$. We change the coordinates $(r_1, \theta_1, r_2, \theta_2)$ to $(r_1, \theta_1, r_2, \theta_W)$ where $\theta_W=\theta_1+\theta_2$. Then the standard symplectic form on $B^4(\sqrt{2})$ is 
    \begin{align*}
        \omega_0&=r_1dr_1\wedge d\theta_1+r_2dr_2\wedge d\theta_2\\
        &= d\left(\frac{r_1^2-r_2^2}{2}\right)\wedge d\theta_1+r_2d r_2\wedge d\theta_W
    \end{align*}
    Since $\displaystyle\frac{r_1^2-r_2^2}{2}=p$, then  $\pi^*\omega_{\mathrm{red}}=\iota^*\omega_0=r_2dr_2\wedge d\theta_W$

    On the other hand we have 
    \begin{align}\label{1}
        \pi^*W^*\left(\frac{1}{2\sqrt{p^2+x^2+y^2}}dx\wedge dy\right)=\frac{r_1r_2^2}{2\sqrt{p^2+r_1^2r_2^2}}dr_1\wedge d\theta_W+\frac{r_1^2r_2}{2\sqrt{p^2+r_1^2r_2^2}}dr_2\wedge d\theta_W
    \end{align}
    Since $r_1^2-r_2^2=2p$, 
    \begin{align*}
        dr_1=\frac{r_2}{\sqrt{2p+r_2^2}}dr_2
    \end{align*}
    Then we can simplify Equation \eqref{1} to
    \begin{align*}
        \pi^*W^*\left(\frac{1}{2\sqrt{p^2+x^2+y^2}}dx\wedge dy\right)=r_2dr_2\wedge d\theta_W
    \end{align*}

    Since $W^*\left(\frac{1}{2\sqrt{p^2+x^2+y^2}}dx\wedge dy\right)=\omega_{\mathrm{red}}$, $W$ is an immersion. Since $\dim G^{-1}(p)/S^1=\dim D_p$, $W$ is an submersion. Thus $W$ is a symplectomorphism.
\end{proof}

Since the $S^1$-action acts freely on $L_1(p, q)$ for $p\ne 0$, then $\pi_p(L_1(p, q))$ is an embedded curve in $D_p$ where $\pi_p:=W\circ\pi: G^{-1}(p)\to D_p$, denoted by $\Gamma_{p, q}$. For $(u_1, u_2)\in L_1(p, q)$, let $w=u_1u_2\in D_p$ . Since $|u_1|^2-|u_2|^2=2p$, then 
\begin{align*}
    |u_1|^2+|u_2|^2=2\sqrt{p^2+|w|^2}
\end{align*}
Thus 
\begin{align*}
\Gamma_{p, q}=\{w\in D_p\mid |1-\sqrt{p^2+|w|^2}+w|=\sqrt{1-q^2}\}
\end{align*}
and 
\begin{align*}
    \pi_p^{-1}(\Gamma_{p, q})=L_1(p, q)
\end{align*}

\begin{lemma}
    For $p\ne 0$, the area enclosed by the curve $\Gamma_{p, q}$ in $\left(D_p, \displaystyle\frac{1}{2\sqrt{p^2+x^2+y^2}}dx\wedge dy\right)$ is $\pi(1-q)$.
\end{lemma}
\begin{proof}
    Using the real coordinates, $w=x+iy$. Let 
    \begin{align*}
        \widetilde{x}=1-\sqrt{p^2+x^2+y^2}+x, \quad \widetilde{y}=y
    \end{align*}
    Then 
    \begin{align*}
        \frac{1}{2\sqrt{p^2+x^2+y^2}}dx\wedge dy=\frac{1}{2(1-\widetilde{x})}d\widetilde{x}\wedge d\widetilde{y}
    \end{align*}
    and the curve $\Gamma_{p, q}$ is 
    \begin{align*}
        \Gamma_{p, q}=\{(\widetilde{x}, \widetilde{y})\in\mathbb{R}\mid \widetilde{x}^2+\widetilde{y}^2=1-q^2\}
    \end{align*}
    By polar coordinates $\widetilde{x}=r\cos(\theta)$ and $\widetilde{y}=r\sin(\theta)$, the area enclosed by $\Gamma_{p, q}$ is 
    \begin{align*}
        \int_0^{2\pi}\int_0^{\sqrt{1-q^2}}\frac{r}{2(1-r\cos(\theta))}drd\theta=\int_0^{\sqrt{1-q^2}}\frac{\pi r}{\sqrt{1-r^2}}dr=\pi(1-q)
    \end{align*}
\end{proof}

\begin{theorem}\label{main1}
    The Lagrangian torus $L(x, y)$ is Hamiltonian isotopic to the toric fiber 
    \begin{align*}
        \begin{cases}
            T\left(x+3y-1, y\right) & \mathrm{for}\:x+2y>1\\
            T\left(y, 1-x-y\right) & \mathrm{for}\: x+2y<1
        \end{cases}
    \end{align*}
\end{theorem}

\begin{proof}
    First we show that $\psi\left((W\circ\pi)^{-1}(S^1(r))\right)$ is Hamiltonian isotopic to a standard toric fiber in $\mathbb{CP}^2(\sqrt{2})$ where $S^1(r)$ is the circle in $D_p$ centered at $(0, 0)$ with radius $r$. Assume $(u_1, u_2)\in (W\circ\pi)^{-1}(S^1(r))$. Then 
    \begin{align*}
        \frac{|u_1|^2-|u_2|^2}{2}=p, \quad |u_1|^2|u_2|^2=r^2
    \end{align*}
    Thus 
    \begin{align*}
        (W\circ\pi)^{-1}(S^1(r))=\left\{\left(\sqrt{p+\sqrt{p^2+r^2}}e^{i\theta_1}, \sqrt{\sqrt{p^2+r^2}-p}e^{i\theta_2}\right)\in B^4(\sqrt{2})\mid 0\le \theta_1\le 2\pi, 0\le \theta_2\le 2\pi\right\}
    \end{align*}
    After changing the coordinates back to $(w_1, w_2)$ on $B^4(\sqrt{2})$, we have 
    {\scriptsize \begin{align*}
        (W\circ\pi)^{-1}(S^1(r))=\left\{\left(\frac{\sqrt{p+\sqrt{p^2+r^2}}e^{i\theta_1}+\sqrt{\sqrt{p^2+r^2}-p}e^{i\theta_2}}{\sqrt{2}}, \frac{\sqrt{\sqrt{p^2+r^2}-p}e^{i\theta_2}-\sqrt{p+\sqrt{p^2+r^2}}e^{i\theta_1}}{\sqrt{2}i}\right)\in B^4(\sqrt{2})\right\}
    \end{align*}}
    Thus 
    {\scriptsize \begin{align*}
        &\psi\left((W\circ\pi)^{-1}(S^1(r))\right)\\
        =&\left\{\left[\sqrt{2-2\sqrt{p^2+r^2}}: \frac{\sqrt{p+\sqrt{p^2+r^2}}e^{i\theta_1}+\sqrt{\sqrt{p^2+r^2}-p}e^{i\theta_2}}{\sqrt{2}} : \frac{\sqrt{\sqrt{p^2+r^2}-p}e^{i\theta_2}-\sqrt{p+\sqrt{p^2+r^2}}e^{i\theta_1}}{\sqrt{2}i}\right]\in\mathbb{CP}^2(\sqrt{2})\right\}
    \end{align*}}
    Since the matrix
    \begin{align*}
        \begin{bmatrix}
            1 & 0 & 0\\
            0 & \frac{1}{\sqrt{2}} & \frac{-i}{\sqrt{2}}\\
            0 & \frac{1}{\sqrt{2}} & \frac{i}{\sqrt{2}}
        \end{bmatrix}\in U(3)
    \end{align*}
    is a Hamiltonian diffeomorphism on $\mathbb{CP}^2(\sqrt{2})$, then $\psi\left((W\circ\pi)^{-1}(S^1(r))\right)$ is Hamiltonian isotopic to 
    \begin{align*}
        \left\{\left[\sqrt{2-2\sqrt{p^2+r^2}}: \sqrt{p+\sqrt{p^2+r^2}}e^{i\theta_1}: \sqrt{\sqrt{p^2+r^2}-p}e^{i\theta_2}\right]\in \mathbb{CP}^2\right\}
    \end{align*}
    which is the standard toric fiber 
    \begin{align*}
    T\left(\frac{p+\sqrt{p^2+r^2}}{2}, \frac{\sqrt{p^2+r^2}-p}{2}\right) 
    \end{align*}
    corresponding to the point 
    \begin{align*}
    \left(p+\sqrt{p^2+r^2}, \sqrt{p^2+r^2}-p\right) 
    \end{align*}
    in the moment polytope. The area of $S^1(r)$ under the symplectic form $\displaystyle\frac{1}{2\sqrt{p^2+x^2+y^2}}dx\wedge dy$ is 
    \begin{align*}
        \pi\left(\sqrt{p^2+r^2}-|p|\right)
    \end{align*}
    Thus when 
    \begin{align*}
    r=\sqrt{(1-q)(1-q+2|p|)}
    \end{align*}
    $S^1(r)$ and $\Gamma_{p, q}$ have the same area. Then $S^1(r)$ and $\Gamma_{p, q}$ are Hamiltonian isotopic in $D_p$ for $r=\sqrt{(1-q)(1-q+2|p|)}$. By Lemma \ref{lem:lifting-reduced-isotopy}, $L_1(p, q)$ and $(W\circ\pi)^{-1}(S^1(r))$ are Hamiltonian isotopic in $B^4(\sqrt{2})$. Since $\psi$ is a symplectomorphism, $\widetilde{L}(p, q)$ and $\psi\left((W\circ\pi)^{-1}(S^1(r))\right)=T\left(\left(p+\sqrt{p^2+r^2}\right)/2, \left(\sqrt{p^2+r^2}-p\right)/2\right)$ are Hamiltonian isotopic in $\psi(B^4(\sqrt{2}))$. Since the Hamiltonian isotopy from Lemma \ref{lem:lifting-reduced-isotopy} in $B^4(\sqrt{2})$ is locally supported, we can trivially extend the Hamiltonian isotopy in $\psi(B^4\sqrt{2})$ to the entire $\mathbb{CP}^2$.

    Since $r=\sqrt{(1-q)(1-q+2|p|)}$, $\widetilde{L}(p, q)$ is Hamiltonian isotopic to toric fiber 
    \begin{align*}
    T\left(\frac{1-q+|p|+p}{2}, \frac{1-q+|p|-p}{2}\right)
    \end{align*}

    \begin{figure}[h]
    \centering
    \includegraphics[scale=0.7]{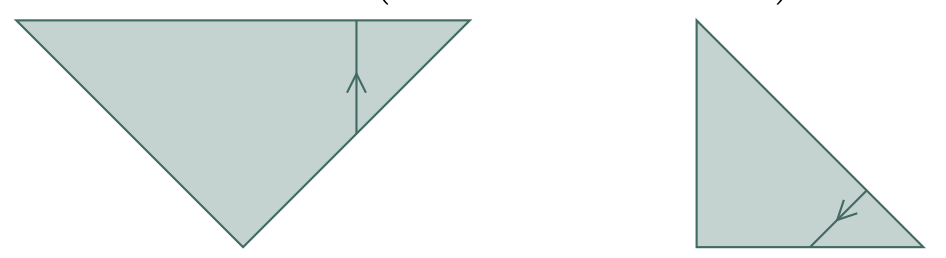}
    \caption{An illustration of the Hamiltonian isotopic Lagrangian tori. Take $p=p_0>0$. Then the fibers over the line $p=p_0$ are Hamiltonian isotopic to those over the line $x-y=p_0$.}
\end{figure}

    Changing the variables back to $x$ and $y$, we have $L(x, y)$ is Hamiltonian isotopic to the toric fiber
    \begin{align*}
        \begin{cases}
            T\left(x+3y-1, y\right) & \mathrm{for}\:x+2y>1\\
            T\left(y, 1-x-y\right) & \mathrm{for}\: x+2y<1
        \end{cases}
    \end{align*}
 
\end{proof}

\begin{remark}
    In the $S^2\times S^2$ case \cite{Lou}, the Lagrangian torus lifting the standard circle is exactly the toric fiber. The situation in the proof of Theorem \ref{main1} is slightly different. Namely the standard circle $S^1(r)$ centered at $(0, 0)$ with radius $r$ does not lift to a Lagrangian torus that is exactly a torus fiber for the moment map defined using the $w$-coordinates. It is instead Hamiltonian isotopic (via an unitary transformation) to a standard toric fiber. Indeed, the pullback of the standard circle is a torus fiber for the moment map defined using the $u$-coordinates, which is related to the $w$-coordinate by a unitary transform.
\end{remark}

%

\section{The wall fibers}
\label{sec:wall-fibers}

The goal of this section is to prove that $L(x, y)$ for $x+2y=1$ is not Hamiltonian isotopic to a product
torus and that distinct wall fibers are pairwise non-Hamiltonian-isotopic. Recall that after changing coordinates $L(x, y)$ for $x+2y=1$ can be written as $\widetilde{L}(0, q)$.

The proof below uses directly the Weinstein-neighborhood argument from
\cite[Section~3]{Lou}.

\begin{proposition}[The unique candidate]
\label{prop:unique-toric-candidate}
If $\widetilde{L}(0, q)$ is Hamiltonian isotopic to a standard toric fiber, then it is Hamiltonian isotopic
to $T\left(\frac{1-q}{2},\frac{1-q}{2}\right)$.
\end{proposition}

\begin{proof}
Suppose that
\[
\phi(\widetilde{L}(0, q))=T(A,B)
\]
for some $\phi\in\operatorname{Ham}(M,\omega_M)$. By Weinstein's Lagrangian neighborhood theorem, there is a symplectomorphism from a neighborhood $U$ of $T(A, B)$ to a neighborhood of the zero section of $T^*T(A, B)$ which takes a Lagrangian torus $C^1$-close to $T(A, B)$ to the image of a closed $1$-form in $T^*T(A, B)$.
For $\varepsilon>0$ sufficiently small, then $\phi(\widetilde{L}(\varepsilon, q))\subset U$ is $C^1$-close to $T(A, B)$. Let $\lambda$ be the $1$-form corresponding to $\phi(\widetilde{L}(\varepsilon, q))$ under the Weinstein's Lagrangian neighborhood theorem. Then there is a standard toric fiber $T(A^\prime, B^\prime)$ such that $[\lambda_{A^\prime, B^\prime}]=[\lambda]\in H^1(T(A, B); \mathbb{R})$ where $(A^\prime, B^\prime)$ is close enough to $T(A, B)$ and $\lambda_{A^\prime, B^\prime}$ is the $1$-form corresponding to $T(A^\prime, B^\prime)$. Thus there is a smooth function $h: T(A, B)\to \mathbb{R}$ such that $\lambda_{A^\prime, B^\prime}-\lambda=dh$. Then we can use $h$ to define a Hamiltonian isotopy from $\widetilde{L}(\varepsilon, q)$ to $T(A^\prime, B^\prime)$.

By Theorem \ref{main1}, we know $\widetilde{L}(\varepsilon, q)$ is Hamiltonian isotopic to $T\left(\frac{1-q+2\varepsilon}{2}, \frac{1-q}{2}\right)$, which has facet-distance multiset 
\begin{align*}
    \{\{\frac{1-q+2\varepsilon}{2}, \frac{1-q}{2}, q-\varepsilon\}\}
\end{align*}
On the other hand, the facet-distance multiset of $T(A^\prime, B^\prime)$ is 
\begin{align*}
    \{\{A^\prime, B^\prime, 1-A^\prime-B^\prime\}\}
\end{align*}
Thus by Theorem \ref{thm:Brendel-classification}
\begin{align*}
    \{\{\frac{1-q+2\varepsilon}{2}, \frac{1-q}{2}, q-\varepsilon\}\}=\{\{A^\prime, B^\prime, 1-A^\prime-B^\prime\}\}
\end{align*}
Take $\varepsilon$ to $0$, we have
\begin{align*}
    \{\{\frac{1-q}{2}, \frac{1-q}{2}, q\}\}=\{\{A, B, 1-A-B\}\}
\end{align*}

Thus the toric fiber $T(A, B)$ is Hamiltonian isotopic to the toric fiber $T\left(\frac{1-q}{2}, \frac{1-q}{2}\right)$.
\end{proof}

\begin{proposition}\label{germ1}
    The displacement energy germ of $\widetilde{L}(0, q)$ is
    \begin{align*}
        S_{\widetilde{L}(0, q)}^e(\epsilon_1, \epsilon_2)=
        \begin{cases}
            \displaystyle\frac{1-q-\epsilon_2}{2} & \mathrm{for}\: q> \displaystyle\frac{1}{3}\\
            q+\epsilon_2-|\epsilon_1| & \mathrm{for}\: q<\displaystyle\frac{1}{3}\\
            \min\left\{\displaystyle\frac{1}{3}-\frac{1}{2}\epsilon_2, \displaystyle\frac{1}{3}+\epsilon_2-|\epsilon_1|\right\} & \mathrm{for}\: q=\displaystyle\frac{1}{3}
        \end{cases}
    \end{align*}
    for sufficiently small $\epsilon_1\ne 0$ and $\epsilon_2$.
\end{proposition}
\begin{proof}

By Theorem \ref{main1}, $\widetilde{L}(\epsilon_1, q+\epsilon_2)$ for $\epsilon_1>0$ is Hamiltonian isotopic to the standard toric fiber $T\left(\frac{1-q-\epsilon_2+2\epsilon_1}{2}, \frac{1-q-\epsilon_2}{2}\right)$ with the facet-distance multiset 
\begin{align*}
        \{\{\frac{1-q-\epsilon_2+2\epsilon_1}{2}, \frac{1-q-\epsilon_2}{2}, q+\epsilon_2-\epsilon_1\}\}
    \end{align*}
and $\widetilde{L}(\epsilon_1, q+\epsilon_2)$ for $\epsilon_1<0$ is Hamiltonian isotopic to the standard toric fiber $T\left(\frac{1-q-\epsilon_2}{2}, \frac{1-q-\epsilon_2-2\epsilon_1}{2}\right)$ with the facet-distance multiset 
\begin{align*}
    \{\{\frac{1-q-\epsilon_2}{2}, \frac{1-q-\epsilon_2-2\epsilon_1}{2}, q+\epsilon_2+\epsilon_1\}\}
\end{align*}

    \textbf{Case 1: $q>\displaystyle\frac{1}{3}$}

    Let $0<\epsilon_1<\frac{3q-1}{10}$ and $|\epsilon_2|<\frac{3q-1}{10}$. We have 
    \begin{align*}
        \frac{1-q-\epsilon_2}{2}<\frac{1-q-\epsilon_2+2\epsilon_1}{2}<q+\epsilon_2-\epsilon_1
    \end{align*}
    By Proposition \ref{dis}, the displacement energy of $\widetilde{L}(\epsilon_1, q+\epsilon_2)$ is
    \begin{align*}
        \min\left\{\frac{1-q-\epsilon_2+2\epsilon_1}{2}, \frac{1-q-\epsilon_2}{2}, q+\epsilon_2-\epsilon_1\right\}=\frac{1-q-\epsilon_2}{2}
    \end{align*}
    Thus 
    \begin{align*}
        S_{\widetilde{L}(0, q)}^e(\epsilon_1, \epsilon_2)=\frac{1-q-\epsilon_2}{2}
    \end{align*}
    for sufficiently small $\epsilon_1>0$.

    Let $\frac{1-3q}{10}<\epsilon_1<0$ and $|\epsilon_2|<\frac{3q-1}{10}$. By the same way as above we have 
    \begin{align*}
        S_{\widetilde{L}(0, q)}^e(\epsilon_1, \epsilon_2)=\frac{1-q-\epsilon_2}{2}
    \end{align*}
    for sufficiently small $\epsilon_1<0$.

    \textbf{Case 2: $q<\displaystyle\frac{1}{3}$}
    
    Let $0<\epsilon_1<\frac{1-3q}{6}$ and $|\epsilon_2|<\frac{1-3q}{6}$. We have 
    \begin{align*}
        q+\epsilon_2-\epsilon_1<\frac{1-q-\epsilon_2}{2}<\frac{1-q-\epsilon_2+2\epsilon_1}{2}
    \end{align*}
    The displacement energy of $\widetilde{L}(\epsilon_1, q+\epsilon_2)$ is $q+\epsilon_2-\epsilon_1$. Thus 
    \begin{align*}
        S_{\widetilde{L}(0, q)}^e(\epsilon_1, \epsilon_2)=q+\epsilon_2-\epsilon_1
    \end{align*}
    for sufficiently small $\epsilon_1>0$

    Let $\frac{3q-1}{6}<\epsilon_1<0$ and $|\epsilon_2|<\frac{1-3q}{6}$. By the same way we have
    \begin{align*}
        S_{\widetilde{L}(0, q)}^e(\epsilon_1, \epsilon_2)=q+\epsilon_2+\epsilon_1
    \end{align*}
    for sufficiently small $\epsilon_1<0$.

    \textbf{Case 3: $q=\displaystyle\frac{1}{3}$}

    For sufficiently small $\epsilon_1>0$, we have
     \begin{align*}
        e_{\CP^2(\sqrt{2})}\left(\widetilde{L}\left(\epsilon_1, \frac{1}{3}+\epsilon_2\right)\right) &= \min\left\{\frac{1}{3}-\frac{1}{2}\epsilon_2+\epsilon_1, \frac{1}{3}-\frac{1}{2}\epsilon_2, \frac{1}{3}+\epsilon_2-\epsilon_1\right\}\\
        &=\min\left\{ \frac{1}{3}-\frac{1}{2}\epsilon_2, \frac{1}{3}+\epsilon_2-\epsilon_1\right\}
    \end{align*}
    Thus 
    \begin{align*}
        S_{\widetilde{L}(0, \frac{1}{3})}^e(\epsilon_1, \epsilon_2)=\min\left\{ \frac{1}{3}-\frac{1}{2}\epsilon_2, \frac{1}{3}+\epsilon_2-\epsilon_1\right\}
    \end{align*}

    For sufficiently small $\epsilon_1<0$, we have 
    \begin{align*}
         e_{\CP^2(\sqrt{2})}\left(\widetilde{L}\left(\epsilon_1, \frac{1}{3}+\epsilon_2\right)\right) &= \min\left\{\frac{1}{3}-\frac{1}{2}\epsilon_2, \frac{1}{3}-\frac{1}{2}\epsilon_2-\epsilon_1, \frac{1}{3}+\epsilon_2+\epsilon_1\right\}\\
         &=\min\left\{\frac{1}{3}-\frac{1}{2}\epsilon_2, \frac{1}{3}+\epsilon_2+\epsilon_1\right\}
    \end{align*}
    Thus 
    \begin{align*}
        S_{\widetilde{L}(0, \frac{1}{3})}^e(\epsilon_1, \epsilon_2)=\min\left\{ \frac{1}{3}-\frac{1}{2}\epsilon_2, \frac{1}{3}+\epsilon_2+\epsilon_1\right\}
    \end{align*}
\end{proof}

\begin{proposition}
    The displacement energy germ of $T\left(\frac{1-q}{2}, \frac{1-q}{2}\right)$ is

    \begin{align*}
        S_{T\left(\frac{1-q}{2}, \frac{1-q}{2}\right)}^e(\epsilon_1, \epsilon_2)=
        \begin{cases}
            \min\left\{\displaystyle\frac{1-q}{2}+\epsilon_1, \frac{1-q}{2}+\epsilon_2\right\} & \mathrm{for}\: q>\displaystyle\frac{1}{3}\\
            q-\epsilon_1-\epsilon_2 & \mathrm{for}\: q<\displaystyle\frac{1}{3}\\
            \min\left\{\frac{1}{3}+\epsilon_1, \frac{1}{3}+\epsilon_2, \frac{1}{3}-\epsilon_1-\epsilon_2\right\} & \mathrm{for}\: q=\displaystyle\frac{1}{3}
        \end{cases}
    \end{align*}
\end{proposition}
\begin{proof}
For sufficiently small $\epsilon_1$ and $\epsilon_2$, the facet-distance multiset of $T\left(\frac{1-q}{2}+\epsilon_1, \frac{1-q}{2}+\epsilon_2\right)$ is
\begin{align*}
    \{\{\frac{1-q}{2}+\epsilon_1, \frac{1-q}{2}+\epsilon_2, q-\epsilon_1-\epsilon_2\}\}
\end{align*}
    \textbf{Case 1: $q>\displaystyle\frac{1}{3}$}

    Let $|\epsilon_1|<\frac{3q-1}{12}$ and $|\epsilon_2|<\frac{3q-1}{12}$. Then 
    \begin{align*}
        \frac{1-q}{2}+\epsilon_1 &< q-\epsilon_1-\epsilon_2\\
        \frac{1-q}{2}+\epsilon_2 &< q-\epsilon_1-\epsilon_2\\
    \end{align*}
    The displacement energy of $T\left(\frac{1-q}{2}+\epsilon_1, \frac{1-q}{2}+\epsilon_2\right)$ is
    \begin{align*}
        e_{\CP^2(\sqrt{2})}\left(T\left(\frac{1-q}{2}+\epsilon_1, \frac{1-q}{2}+\epsilon_2\right)\right)=\min\left\{\frac{1-q}{2}+\epsilon_1, \frac{1-q}{2}+\epsilon_2\right\}
    \end{align*}
    Thus
    \begin{align*}
        S_{T\left(\frac{1-q}{2}, \frac{1-q}{2}\right)}^e(\epsilon_1, \epsilon_2)=\min\left\{\frac{1-q}{2}+\epsilon_1, \frac{1-q}{2}+\epsilon_2\right\}
    \end{align*}

    \textbf{Case 2: $q<\displaystyle\frac{1}{3}$}

    Let $|\epsilon_1|<\frac{1-3q}{12}$ and $|\epsilon_2|<\frac{1-3q}{12}$. Then 
    \begin{align*}
        \frac{1-q}{2}+\epsilon_1 &> q-\epsilon_1-\epsilon_2\\
        \frac{1-q}{2}+\epsilon_2 &> q-\epsilon_1-\epsilon_2\\
    \end{align*}
    The displacement energy of $T\left(\frac{1-q}{2}+\epsilon_1, \frac{1-q}{2}+\epsilon_2\right)$ is
    \begin{align*}
        e_{\CP^2(\sqrt{2})}\left(T\left(\frac{1-q}{2}+\epsilon_1, \frac{1-q}{2}+\epsilon_2\right)\right)=q-\epsilon_1-\epsilon_2
    \end{align*}
     Thus
    \begin{align*}
        S_{T\left(\frac{1-q}{2}, \frac{1-q}{2}\right)}^e(\epsilon_1, \epsilon_2)=q-\epsilon_1-\epsilon_2
    \end{align*}

    \textbf{Case 3: $q=\displaystyle\frac{1}{3}$}

    The displacement energy of $T\left(\frac{1}{3}+\epsilon_1, \frac{1}{3}+\epsilon_2\right)$ is
    \begin{align*}
        e_{\CP^2(\sqrt{2})}\left(T\left(\frac{1}{3}+\epsilon_1, \frac{1}{3}+\epsilon_2\right)\right)=\min\left\{\frac{1}{3}+\epsilon_1, \frac{1}{3}+\epsilon_2, \frac{1}{3}-\epsilon_1-\epsilon_2\right\}
    \end{align*}
    Thus 
    \begin{align*}
        S_{T\left(\frac{1}{3}, \frac{1}{3}\right)}^e(\epsilon_1, \epsilon_2)=\min\left\{\frac{1}{3}+\epsilon_1, \frac{1}{3}+\epsilon_2, \frac{1}{3}-\epsilon_1-\epsilon_2\right\}
    \end{align*}
\end{proof}

\begin{theorem}[No wall fiber is a standard fiber]
\label{thm:no-wall-standard-fiber}
The wall fiber \(\widetilde{L}(0, q)\) is not Hamiltonian
isotopic to any standard toric fiber.
\end{theorem}

\begin{proof}
Assume $\widetilde{L}(0, q)$ is Hamiltonian isotopic to a standard toric fiber. By Proposition \ref{prop:unique-toric-candidate}, the only possibility is
\(T\left(\frac{1-q}{2},\frac{1-q}{2}\right)\). 
By Definition \ref{def:displacement-energy-germ}, 
\begin{align*}
    S_{\widetilde{L}(0, q)}^e=S_{T\left(\frac{1-q}{2},\frac{1-q}{2}\right)}^e\circ A
\end{align*}
for some \(A\in\mathrm{GL}(2,\mathbb Z)\). We rule out such an
identity even for \(A\in\mathrm{GL}(2,\mathbb R)\).

\textbf{Case 1: $q>\displaystyle\frac{1}{3}$}

The displacement energy germ $S_{\widetilde{L}(0, q)}^e$ is given by a single linear function, but the displacement energy germ $S_{T\left(\frac{1-q}{2},\frac{1-q}{2}\right)}^e$ is given by the minimum of two linearly independent functions, each of which achieves minimum in some region. Thus the two germs cannot be related by a linear map.

\textbf{Case 2: $q<\displaystyle\frac{1}{3}$}

Write
\[
A(\epsilon_1,\epsilon_2)
=
\bigl(a(\epsilon_1,\epsilon_2),
      b(\epsilon_1,\epsilon_2)\bigr),
\]
where \(a\) and \(b\) are independent linear functions.

Choose \(\delta>0\) sufficiently small and approach the origin in the
two opposite directions
\[
z=(\delta,0),
\qquad
-z=(-\delta,0).
\]
Then 
\begin{align*}
    S_{\widetilde{L}(0, q)}^e(\delta, 0) &=q-\delta\\
    S_{\widetilde{L}(0, q)}^e(-\delta, 0) &=q-\delta
\end{align*}
Hence
\begin{align*}
    S_{\widetilde{L}(0, q)}^e(\delta, 0)+S_{\widetilde{L}(0, q)}^e(-\delta, 0)=2q-2\delta
\end{align*}

On the other hand 
\begin{align*}
    S_{T\left(\frac{1-q}{2}, \frac{1-q}{2}\right)}^e(Az)&=q-a(z)-b(z)\\
    S_{T\left(\frac{1-q}{2}, \frac{1-q}{2}\right)}^e(A(-z))&=q-a(-z)-b(-z)=q+a(z)+b(z)
\end{align*}
Therefore
\begin{align*}
    S_{T\left(\frac{1-q}{2}, \frac{1-q}{2}\right)}^e(Az)+S_{T\left(\frac{1-q}{2}, \frac{1-q}{2}\right)}^e(A(-z))=2q
\end{align*}
contradicting the assumed identity of germs.

\textbf{Case 3: $q=\displaystyle\frac{1}{3}$}
After subtracting the common constant $\dfrac{1}{3}$, the continuous extension of the wall formula is the minimum of
\[
\ell_1=-\frac12\epsilon _2,\qquad\ell_2=\epsilon _2-\epsilon _1,\qquad\ell_3=\epsilon _2+\epsilon _1,
\]
whereas the Clifford-torus formula is the minimum of
\[
m_1=\epsilon _1,\qquad m_2=\epsilon _2,\qquad m_3=-\epsilon _1-\epsilon _2.
\]
Each of the three functions in each list realizes the minimum on a nonempty open region.
But
\[
m_1+m_2+m_3=0
\]
hence
\[
m_1\circ A+m_2\circ A+m_3\circ A=0,
\]
 whereas
\[
\ell_1+\ell_2+\ell_3=\frac32\epsilon _2\ne0,
\]
\end{proof}

\begin{corollary}[Pairwise distinction on the wall]
\[
    \widetilde L(0,q)\simeq_{\mathrm{Ham}}\widetilde L(0,q')
    \qquad\Longleftrightarrow\qquad
    q=q'.
\]
\end{corollary}

\begin{proof}
Only the forward implication requires proof. Suppose that
\[
    \widetilde L(0,q)\simeq_{\mathrm{Ham}}\widetilde L(0,q').
\]
Naturality of the displacement-energy germ gives
\[
    S^e_{\widetilde L(0,q)}
    =
    S^e_{\widetilde L(0,q')}\circ A
\]
for some invertible linear map
\[
    A(\epsilon_1,\epsilon_2)
    =
    \bigl(a(\epsilon_1,\epsilon_2),
          b(\epsilon_1,\epsilon_2)\bigr).
\]

Choose
\[
    \epsilon=(\epsilon_1,\epsilon_2)\in\mathbb R^2
\]
such that
\[
    \epsilon_1\neq0,
    \qquad
    a(\epsilon)\neq0.
\]
Such an \(\epsilon\) exists because \(a\) is a nonzero linear functional.
For every sufficiently small \(t>0\), Proposition~4.2 applies both to
\(\pm t\epsilon\) and to
\[
    A(\pm t\epsilon)=\pm tA(\epsilon).
\]

The formulas in Proposition~4.2 give
\[
\lim_{t\to0^+}
S^e_{\widetilde L(0,q)}(t\epsilon)
=
c(q),
\]
where
\[
c(q)=
\begin{cases}
q, & q<\dfrac13,\\[2mm]
\dfrac13, & q=\dfrac13,\\[2mm]
\dfrac{1-q}{2}, & q>\dfrac13.
\end{cases}
\]
Similarly,
\[
\lim_{t\to0^+}
S^e_{\widetilde L(0,q')}\bigl(tA(\epsilon)\bigr)
=
c(q'),
\]
where \(c(q')\) is given by the same formula with \(q\) replaced by
\(q'\). Therefore, taking \(t\to0^+\) in
\[
S^e_{\widetilde L(0,q)}(t\epsilon)
=
S^e_{\widetilde L(0,q')}\bigl(tA(\epsilon)\bigr)
\]
gives
\[
    c(q)=c(q').
\]

If \(q,q'<1/3\), this equality gives \(q=q'\). If \(q,q'>1/3\), it
gives
\[
    \frac{1-q}{2}=\frac{1-q'}{2},
\]
and hence again \(q=q'\). Moreover, if one of \(q,q'\) equals \(1/3\),
then \(c(q)=c(q')\) implies that the other also equals \(1/3\).

It remains to exclude the possibility that \(q\) and \(q'\) lie on
opposite sides of \(1/3\). After interchanging them if necessary,
suppose that
\[
    q<\frac13<q'.
\]
For every sufficiently small \(t>0\), Proposition~4.2 gives
\[
\begin{aligned}
&S^e_{\widetilde L(0,q)}(t\epsilon)
 +S^e_{\widetilde L(0,q)}(-t\epsilon)\\
&\qquad
=
\bigl(q+t\epsilon_2-t|\epsilon_1|\bigr)
+
\bigl(q-t\epsilon_2-t|\epsilon_1|\bigr)\\
&\qquad
=
2q-2t|\epsilon_1|.
\end{aligned}
\]
On the other hand,
\[
\begin{aligned}
&S^e_{\widetilde L(0,q')}\bigl(tA(\epsilon)\bigr)
 +S^e_{\widetilde L(0,q')}\bigl(-tA(\epsilon)\bigr)\\
&\qquad
=
\frac{1-q'-tb(\epsilon)}{2}
+
\frac{1-q'+tb(\epsilon)}{2}\\
&\qquad
=
1-q'.
\end{aligned}
\]
Naturality therefore implies
\[
    2q-2t|\epsilon_1|=1-q'
\]
for every sufficiently small \(t>0\). The right-hand side is independent
of \(t\), whereas the left-hand side depends nontrivially on \(t\) unless
\(|\epsilon_1|=0\). Hence \(\epsilon_1=0\), contradicting our choice of
\(\epsilon\). Therefore \(q\) and \(q'\) cannot lie on opposite sides
of \(1/3\), and hence \(q=q'\).
\end{proof}

The Main Theorem follows from Theorem \ref{main1} and Theorem \ref{thm:no-wall-standard-fiber}

\begin{remark}[The monotone fiber]
\label{rem:monotone-wall-fiber}
At \(q=1/3\), the fiber \(L_t\) is Wu's monotone fiber
\cite[Propositions~3.3 and~7.5]{OakleyUsher}, and Oakley--Usher prove that it
is Hamiltonian isotopic to the Chekanov--Schlenk torus
\cite[Theorem~1.2]{OakleyUsher}. The last case of
Theorem~\ref{thm:no-wall-standard-fiber} distinguishes it from the
Clifford torus directly by comparing the formulas obtained by
approaching the origin from the three different regions. Alternatively,
the monotone case follows from this identification together with
\cite[Theorem~1 and Section~7]{ChekanovSchlenk}.
\end{remark}

\bibliographystyle{alpha} 
\bibliography{refer}

\end{document}